\documentclass[11pt]{amsart} \textwidth=14.5cm \oddsidemargin=1cm
\usepackage[rgb]{xcolor}
\usepackage{tikz}
\usepackage{verbatim}
\usepackage{amsmath}
\usepackage{amsxtra}
\usepackage{amscd}
\usepackage{amsthm}
\usepackage{amsfonts}
\usepackage{amssymb}
\usepackage{eucal}
\usepackage{float}
\usepackage{mathrsfs}
\usepackage{graphicx}
\usepackage{stmaryrd}
\usepackage{tikz-cd}
\usetikzlibrary{arrows}
\usetikzlibrary{cd}
\usepackage{ytableau}

\usepackage{latexsym,todonotes}

\usepackage[linktocpage=true]{hyperref}
\hypersetup{colorlinks,linkcolor=blue,urlcolor=orange,citecolor=blue}
\usepackage[all, cmtip]{xy}

\usepackage{stmaryrd}
\usepackage{color} 

\usepackage{amsrefs}

\newtheorem{thm}{Theorem} [section]

\newtheorem{lem}[thm]{Lemma}
\newtheorem{prop}[thm]{Proposition}

\newtheorem{problem}[thm]{Problem}

\theoremstyle{definition}
\newtheorem{defn}[thm]{Definition}

\theoremstyle{remark}
\newtheorem{rem}[thm]{Remark}

\numberwithin{equation}{section}

\begin{document}

\newcommand{\thmref}[1]{Theorem~\ref{#1}}
\newcommand{\secref}[1]{Section~\ref{#1}}
\newcommand{\lemref}[1]{Lemma~\ref{#1}}
\newcommand{\propref}[1]{Proposition~\ref{#1}}
\newcommand{\corref}[1]{Corollary~\ref{#1}}
\newcommand{\remref}[1]{Remark~\ref{#1}}
\newcommand{\eqnref}[1]{(\ref{#1})}

\newcommand{\exref}[1]{Example~\ref{#1}}

\newtheorem{innercustomthm}{{\bf Theorem}}
\newenvironment{customthm}[1]
  {\renewcommand\theinnercustomthm{#1}\innercustomthm}
  {\endinnercustomthm}
  
  \newtheorem{innercustomcor}{{\bf Corollary}}
\newenvironment{customcor}[1]
  {\renewcommand\theinnercustomcor{#1}\innercustomcor}
  {\endinnercustomthm}
  
  \newtheorem{innercustomprop}{{\bf Proposition}}
\newenvironment{customprop}[1]
  {\renewcommand\theinnercustomprop{#1}\innercustomprop}
  {\endinnercustomthm}

\newcommand{\bbinom}[2]{\begin{bmatrix}#1 \\ #2\end{bmatrix}}
\newcommand{\cbinom}[2]{\set{\^!\^!\^!\begin{array}{c} #1 \\ #2\end{array}\^!\^!\^!}}
\newcommand{\abinom}[2]{\ang{\^!\^!\^!\begin{array}{c} #1 \\ #2\end{array}\^!\^!\^!}}
\newcommand{\qfact}[1]{[#1]^^!}

\newcommand{\nc}{\newcommand}

\nc{\Ord}{\text{Ord}_v}

 \nc{\A}{\mathcal A} 
  \nc{\G}{\mathbb G} 
\nc{\Ainv}{\A^{\rm inv}}
\nc{\aA}{{}_\A}
\nc{\aAp}{{}_\A'}
\nc{\aff}{{}_\A\f}
\nc{\aL}{{}_\A L}
\nc{\aM}{{}_\A M}
\nc{\Bin}{B_i^{(n)}}
\nc{\dL}{{}^\omega L}
\nc{\Z}{{\mathbb Z}}
 \nc{\C}{{\mathbb C}}
 \nc{\N}{{\mathbb N}}
 \nc{\R}{{\mathbb R}}
  \renewcommand{\S}{{\mathcal S}}
  \nc{\I}{{\mathcal I}}
 \nc{\fZ}{{\mf Z}}
 \nc{\F}{{\mf F}}
 \nc{\Q}{\mathbb{Q}}
 \nc{\la}{\lambda}
 \nc{\ep}{\epsilon}
 \nc{\h}{\mathfrak h}
 \nc{\He}{\bold{H}}
 \nc{\htt}{\text{tr }}
 \nc{\n}{\mf n}
 \nc{\g}{{\mathfrak g}}
 \nc{\DG}{\widetilde{\mathfrak g}}
 \nc{\SG}{\breve{\mathfrak g}}
 \nc{\is}{{\mathbf i}}
 \nc{\V}{\mf V}
 \nc{\bi}{\bibitem}
 \nc{\E}{\mc E}
 \nc{\ba}{\tilde{\pa}}
 \nc{\half}{\frac{1}{2}}
 \nc{\hgt}{\text{ht}}
 \nc{\ka}{\kappa}
 \nc{\mc}{\mathcal}
 \nc{\mf}{\mathfrak} 
 \nc{\hf}{\frac{1}{2}}
\nc{\ov}{\overline}
\nc{\ul}{\underline}

\nc{\xx}{{\mf x}}
\nc{\id}{\text{id}}
\nc{\one}{\bold{1}}
\nc{\mfsl}{\mf{sl}}
\nc{\mfgl}{\mf{gl}}
\nc{\ti}[1]{\textit{#1}}
\nc{\Hom}{\text{Hom}}
\nc{\Cat}{\mathscr{C}}
\nc{\CatO}{\mathscr{O}}
\renewcommand{\O}{\mathscr{O}}
\nc{\Tan}{\mathscr{T}}
\nc{\Umod}{\mathscr{U}}
\nc{\Func}{\mathscr{F}}
\nc{\Kh}{\text{Kh}}
\nc{\Khb}[1]{\llbracket #1 \rrbracket}

\nc{\ua}{\mf{u}}
\nc{\nb}{u}
\nc{\inv}{\theta}
\nc{\mA}{\mathcal{A}}
\newcommand{\TT}{\mathbf T}
\newcommand{\TA}{{}_\A{\TT}}
\newcommand{\tK}{\widetilde{K}}
\newcommand{\al}{\alpha}
\newcommand{\Fr}{\bold{Fr}}

\nc{\Qq}{\Q(v)}
\nc{\U}{\bold{U}}
\nc{\uu}{\mathfrak{u}}
\nc{\Udot}{\dot{\U}}
\nc{\f}{\bold{f}}
\nc{\fprime}{\bold{'f}}
\nc{\B}{\bold{B}}
\nc{\Bdot}{\dot{\B}}
\nc{\Dupsilon}{\Upsilon^{\vartriangle}}
\newcommand{\T}{\texttt T}
\newcommand{\vs}{\varsigma}
\newcommand{\Pa}{{\bf{P}}}
\newcommand{\Padot}{\dot{\bf{P}}}

\nc{\ipsi}{\psi_{\imath}}
\nc{\Ui}{{\bold{U}^{\imath}}}
\nc{\uidot}{\dot{\mathfrak{u}}^{\imath}}
\nc{\Uidot}{\dot{\bold{U}}^{\imath}}
 \nc{\be}{e}
 \nc{\bff}{f}
 \nc{\bk}{k}
 \nc{\bt}{t}
 \nc{\bs}{\backslash}
 \nc{\BLambda}{{\Lambda_{\inv}}}
\nc{\Ktilde}{\widetilde{K}}
\nc{\bktilde}{\widetilde{k}}
\nc{\Yi}{Y^{w_0}}
\nc{\bunlambda}{\Lambda^\imath}
\newcommand{\Iwhite}{\I_{\circ}}
\nc{\ile}{\le_\imath}
\nc{\il}{<_{\imath}}

\newcommand{\ff}{B}


\nc{\etab}{\eta^{\bullet}}
\newcommand{\Iblack}{\I_{\bullet}}
\newcommand{\wb}{w_\bullet}
\newcommand{\UIblack}{\U_{\Iblack}}

\newcommand{\blue}[1]{{\color{blue}#1}}
\newcommand{\red}[1]{{\color{red}#1}}
\newcommand{\green}[1]{{\color{green}#1}}
\newcommand{\white}[1]{{\color{white}#1}}

\newcommand{\dvd}[1]{t_{\odd}^{{(#1)}}}
\newcommand{\dvp}[1]{t_{\ev}^{{(#1)}}}
\newcommand{\ev}{\mathrm{ev}}
\newcommand{\odd}{\mathrm{odd}}

\newcommand\TikCircle[1][2.5]{{\mathop{\tikz[baseline=-#1]{\draw[thick](0,0)circle[radius=#1mm];}}}}

\newcommand{\commentcustom}[1]{}

\raggedbottom

\title{the twisted gan-gross-prasad problem for finite classical groups}
\author{nhat hoang le}
\begin{abstract}
In this paper, we study the twisted Gan-Gross-Prasad problem for classical
groups over finite fields. We formulate a multiplicity formula for
Deligne-Lusztig characters and give a complete answer for cuspidal
representations arising from elliptic tori.
\end{abstract}

\address{Department of Mathematics, National University of Singapore, Block
S17, 10 Lower Kent Ridge Road, Singapore 119076.}
\email{lnhoang@nus.edu.sg}
\maketitle

\section{Introduction}\label{sec1}

In \cite{GP1,GP2}, B. Gross and D. Prasad studied the problem of
restriction of irreducible representations of special orthogonal groups
over a local field based on the Langlands parametrization of irreducible representations. After that, joint with W. T. Gan, in \cite{GGP1}, they formulated the conjecture to all classical groups, which is called the local Gan-Gross-Prasad (GGP) conjecture. In the $p$-adic case, the local Gan-Gross-Prasad conjecture has been solved by J.-L. Waldspurger and C. Moeglin and J.-L. Waldspurger \cite{MW,Wal1,Wal2,Wal3} for orthogonal groups, by R. Beuzart-Plessis \cite{BP1,BP2} for Bessel models for unitary groups and W. T. Gan and A. Ichino \cite{GI} for Fourier-Jacobi models for unitary groups, and by H. Atobe \cite{Ato} for symplectic-metaplectic groups. In the finite field case, the Gan-Gross-Prasad problem for finite classical groups has been studied by D. Liu and Z. Wang, and Z. Wang in \cite{LW1,LW2,LW3,Wang1,Wang2}.

Recently, in \cite{GGP3}, W. T. Gan and B. Gross and D. Prasad consider
a twisted variant of the local Gan-Gross-Prasad conjecture (for Fourier-Jacobi
models). Let $K$ be a nonarchimedean local field and $L$ be a separable
quadratic extension of $K$. Let $V$ be a non-degenerate skew-hermitian
space of dimension $n$ over $L$. Let $\omega_{V,\psi,\mu}$ be the
Weil representation of the isometry group $\text{U}\left(V\right)$,
where $\psi$ is a nontrivial additive character of $K$ and $\mu$
is a conjugate-symplectic character of $L^{\times}$ such that $\mu\mid_{K^{\times}}$
is the quadratic character $\omega_{L/K}$ associated to $L/K$ by
local class field theory. Instead of considering $\text{U}\left(V\right)$
as a subgroup of $\text{U}\left(V\right)\left(K\times K\right)=\text{U}\left(V\right)\times\text{U}\left(V\right)$,
they consider it as a subgroup of $\text{U}\left(V\right)\left(L\right)\cong\text{GL}_{n}\left(L\right)$.
They propose the problem of determining 
\[
\dim\text{Hom}_{\text{U}\left(V\right)}\left(\Pi,\omega_{V,\psi,\mu}\right),
\]
where $\Pi$ is an irreducible generic representation of $\text{GL}_{n}\left(L\right)$.
In \cite{Le1,Le2}, the author established the tempered case of the conjecture.

In section 2.4 in \cite{GGP3}, the three authors also consider restriction
problems for both skew-hermitian spaces and symplectic groups over
a finite field $k$. In this paper, we study the twisted
Gan-Gross-Prasad problems over finite fields. Let $F$ be the Frobenius automorphism in $\text{Gal}(\bar{k}/k)$ and $k^\prime$ be the quadratic extension of $k$. We want to determine $\dim\text{Hom}_{\text{U}_{n}\left(k\right)}\left(\pi,\omega_{_{\psi}}\right)$
or $\dim\text{Hom}_{\text{Sp}_{2n}\left(k\right)}\left(\pi,\omega_{_{\psi}}\right)$,
where $\pi$ is an irreducible representation of $\text{GL}_{n}\left(k^\prime\right)$
or $\text{Sp}_{2n}\left(k^\prime\right)$, respectively,
and $\omega_{\psi}$ is the Weil representation of $\text{Sp}_{2n}\left(k\right)$
(one can restrict it to $\text{U}_{n}\left(k\right)$
if necessary) associated to a nontrivial additive character $\psi$
of $k$. Noting that we follow the construction of the
Weil representation of $\text{Sp}_{2n}\left(k\right)$
in \cite{Howe}. Motivated by \cite{Ree}, \cite{LMS} and \cite{Shi2}, we first give a multiplicity formula for $\left\langle R_{T,\chi}^{G},\omega_{\psi}\right\rangle _{H^{F}}$,
where $R_{T,\chi}^{G}$ is a Deligne-Lusztig virtual character. The
following is our first main result.
\begin{thm}\label{thm1.1}
We have the following multiplicity 
\[
\left\langle R_{T,\chi}^{G},\omega_{\psi}\right\rangle _{H^{F}}=\sum_{\kappa\in \bar{K}_H(S)^F}\sum_{\jmath\in \mathcal{J}(T,Z_\kappa)^F}
\frac{(-1)^{\text{rk}_{k^\prime}(T)+\text{rk}_{k^\prime}(G)+l(Z_\kappa)}}
{|W_G(T)^F|}\sum_{\gamma \in W_G(T)^F}\langle ^\gamma\chi_\kappa,\vartheta_{Z_\kappa}\rangle_{Z_\kappa^{F}},
\]
where $S$ is an $F$-stable maximal torus of $H$.
\end{thm}

Using the above theorem and by an explicit computation, we are able
to determine the above multiplicities for cuspidal Deligne-Lusztig representations in the case when $(G,H)=(\text{Res}_{k^\prime/k}\text{GL}_n,\text{U}_n)$.
\begin{thm}\label{thm1.2}
Let $T$ be an $F$-stable maximal torus of $G$ which is anisotropic
mod $Z\left(G\right)$. Then  
$$
\langle(-1)^{\text{rk}_{k^\prime}T+\text{rk}_{k^\prime}G}R^G_{T,\chi},\omega_\psi\rangle_{H^F}=
\begin{cases}
1 & \text{if }\chi|_{k_n^\times} \notin \Gamma \cdot \vartheta_n,
\\
2 & \text{otherwise}
\end{cases}
$$
when $n$ is even and
$$
\langle(-1)^{\text{rk}_{k^\prime}T+\text{rk}_{k^\prime}G}R^G_{T,\chi},\omega_\psi\rangle_{H^F}=
\begin{cases}
1 & \text{if }\chi|_{k^1_{2n}} \notin \Gamma \cdot \vartheta^\prime_n,
\\
0 & \text{otherwise}
\end{cases}
$$
when $n$ is odd.
\end{thm}

This paper is organized as follows. In Section \ref{sec2}, we recall necessary materials, including Deligne-Lusztig characters, character formulae, the Weil representation and its geometrization. In Section \ref{sec3}, we prove our main result, which is Theorem \ref{thm1.1}. Finally, in Section \ref{sec4}, we give some explicit computations (Theorem \ref{thm4.2}, Theorem \ref{thm4.3} and Theorem \ref{thm4.4}) as applications of Theorem \ref{thm1.1}.
\subsection{Acknowledgement}
The author would like to thank his supervisor Professor Wee Teck Gan
for suggesting this problem and many useful advices. He also thanks
Bryan Wang Peng Jun and Jialiang Zou for useful comments. The author
is supported by NUS President's Graduate Fellowship.

\section{Preliminaries}\label{sec2}

\subsection{Deligne-Lusztig characters.}

In this subsection, we follow the general framework in \cite{LMS}
and \cite{Ree}. Let $G$ be a connected reductive algebraic group
over $k$. Let $F$ be the Frobenius endomorphism of
$\text{Gal}\left(\bar{k}/k\right)$. Let
$T$ be an $F$-stable maximal torus of $G$. Let $W_{G}\left(T\right)=N_{G}\left(T\right)/T$
be the Weyl group of $G$. By Lang-Steinberg theorem, we have $W_{G}\left(T\right)^{F}=N_{G}\left(T\right)^{F}/T^{F}$.
Let $s\in G^{F}$ be a semisimple element , and $G_{s}$ be the identity
component of the centralizer of $s$ in $G$. Let 
\[
N_{G}\left(s,T\right)^{F}=\left\{ \gamma\in G^{F}:\gamma^{-1}s\gamma\in T\right\} .
\]
Then $G_{s}^{F}\times N_{G}\left(T\right)^{F}$ acts on $N_{G}\left(s,T\right)^{F}$
and we denote 
\[
\bar{N}_{G}\left(s,T\right)^{F}=G_{s}^{F}\backslash N_{G}\left(s,T\right)^{F}.
\]
Let $W_{G}=W_{G}\left(T_{0}\right)$, where $T_{0}$ is an $F$-stable
maximal torus in $G$ contained in an $F$-stable Borel subgroup of
$G$. We can associate to $T$ a cohomology class $\text{cl}\left(T,G\right)$ in $H^{1}\left(F,W_{G}\right)$.
Similarly, we define $W_{G_{s}}$. As in \cite{Ree}, we have
the following map 
\[
j_{G_{s}}:H^{1}\left(F,W_{G_{s}}\right)\rightarrow H^{1}\left(F,W_{G}\right).
\]
Let $\mathcal{T}\left(G\right)$ be the set of all $F$-stable maximal
tori of $G$, and for $\omega\in H^{1}\left(F,W_{G}\right)$, let
\[
\mathcal{T}_{\omega}\left(G\right)=\left\{ T\in\mathcal{T}\left(G\right):\text{cl}\left(T,G\right)=\omega\right\} .
\]
For $T\in\mathcal{T}_{\omega}\left(G\right)$, the set $N_{G}\left(s,T\right)^{F}$
is nonempty if and only if $j_{G_{s}}^{-1}\left(\omega\right)\neq\emptyset$.
In this case, we have 
\[
\left|\bar{N}_{G}\left(s,T\right)^{F}\right|=\underset{v\in j_{G_{s}}^{-1}\left(\omega\right)}{\sum}\frac{\left|W_{G}\left(T\right)^{F}\right|}{\left|W_{G_{s}}\left(T_{v}\right)^{F}\right|},
\]
where $T_{v}$ is an arbitrary element in $\mathcal{T}_{v}\left(G_{s}\right)$,
for each $v\in j_{G_{s}}^{-1}\left(\omega\right)$.

In \cite{DL}, to parametrize irreducible representations of $G$,
Deligne and Lusztig constructed a generalized character $R^G_{T,\chi}$
of $G^{F}$ associated to an $F$-stable maximal torus $T$ of $G$
and a character $\chi$ of $T^{F}$. Let $g\in G^{F}$ with the Jordan
decomposition $g=su$. Let $\omega=\text{cl}\left(T,G\right)$. Following
\cite{DL}, we have the reduction formula 
\[
R_{T,\chi}^{G}\left(su\right)=\underset{\bar{\gamma}\in\bar{N}_{G}\left(s,T\right)^{F}}{\sum}{}^{\gamma}\chi\left(s\right)Q_{^{\gamma}T}^{G_{s}}\left(u\right),
\]
where $\gamma\in N_{G}\left(s,T\right)^{F}$ is a representative of
$\bar{\gamma}$, $^{\gamma}T=\gamma T\gamma^{-1}$, $^{\gamma}\chi=\chi\circ\text{Ad}\left(\gamma^{-1}\right)$
is a character of $^{\gamma}T^{F}$ and $Q_{^{\gamma}T}^{G_{s}}$
is the Green function. Moreover, by decomposing $\bar{N}_{G}\left(s,T\right)^{F}$
into its $W_{G}\left(T\right)^{F}$-orbits, we have 
\[
R_{T,\chi}^{G}\left(su\right)=\underset{v\in j_{G_{s}}^{-1}\left(\omega\right)}{\sum}\chi_{v}\left(s\right)Q_{T_{v}}^{G_{s}}\left(u\right),
\]
where $\mathcal{O}_{v}$ is the $W_{G}\left(T\right)^{F}$-orbit in
$\bar{N}_{G}\left(s,T\right)^{F}$ corresponding to $v$ and $
\chi_{v}=\underset{\bar{\gamma}\in\mathcal{O}_{v}}{\sum}{}^{\gamma}\chi$, which is a well-defined function on $Z\left(G_{s}\right)^{F}$. Moreover, the term
$\chi_{v}\left(s\right)$ can be written as 
\[
\chi_{v}\left(s\right)=\frac{1}{\left|W_{G_{s}}\left(T_{v}\right)^{F}\right|}\underset{x\in W_{G}\left(T\right)^{F}}{\sum}{}^{\gamma x}\chi\left(s\right),
\]
where $\gamma$ is an arbitrary element of $N_{G}\left(s,T\right)^{F}$
such that $\bar{\gamma}\in\mathcal{O}_{v}$.

\subsection{Multiplicity formula.}\label{sec2.2}

Let $T\in\mathcal{T}\left(G\right)$. Assume that $f:G^{F}\rightarrow\mathbb{C}$
is a virtual character supported on $
\left\{ g\in G^{F}:\,g=su\text{ and }\text{Ad}\left(G^{F}\right)\cdot s\cap T\neq\emptyset\right\}$. For each $s\in T^{F}$, let $G_{s}^{\text{upt}}$ be the set of unipotent
elements of $G_{s}$ and $\mathcal{U}\left(G_{s}^{F}\right)$ be the
finite set of $\text{Ad}\left(G_{s}^{F}\right)$-orbits in $\left(G_{s}^{\text{upt}}\right)^{F}$,
where $G_{s}^{\text{upt}}$ is the set containing unipotent element
of $G_{s}$. We have 
\[
\frac{1}{\left|G^{F}\right|}\underset{g\in G^{F}}{\sum}f\left(g\right)=\frac{1}{\left|G^{F}\right|}\underset{s\in G_{ss}^{F}}{\sum}\,\underset{u\in\left(G_{s}^{\text{upt}}\right)^{F}}{\sum}f\left(su\right)
\]
\[
=\frac{1}{\left|G^{F}\right|}\underset{s\in T^{F}}{\sum}\frac{\left|\text{Ad}\left(G^{F}\right)\cdot s\right|}{\left|\text{Ad}\left(G^{F}\right)\cdot s\cap T\right|}\underset{\left[u\right]\in{\mathcal{U}}\left(G_{s}^{F}\right)}{\sum}\left|\text{Ad}\left(G_{s}^{F}\right)\cdot u\right|f\left(su\right).
\]
Since $\text{Ad}\left(G^{F}\right)\cdot s\cap T\simeq C_{G}\left(s\right)^{F}\backslash N_{G}\left(s,T\right)^{F}$,
it follows that 
\[
\frac{1}{\left|G^{F}\right|}\underset{g\in G^{F}}{\sum}f\left(g\right)=\underset{s\in T^{F}}{\sum}\frac{1}{\left|\bar{N}_{G}\left(s,T\right)^{F}\right|}\underset{\left[u\right]\in\mathcal{U}\left(G_{s}^{F}\right)}{\sum}\frac{1}{\left|C_{G_{s}}\left(u\right)^{F}\right|}f\left(su\right).
\]
Let $I\left(T\right)$ be an index set for the set of subgroups $\left\{ G_{s}:s\in T\right\} $.
For $\iota\in I\left(T\right)$, let $G_{\iota}$ be the corresponding
identity component of centralizer and $T_{\iota}=\left\{ s\in T:G_{s}=G_{\iota}\right\}$. Let $H$ be a subgroup of $G$ defined over $k$ and $S$ be an $F$-stable maximal torus of $H$. For any $r\in G(k)$, we denote $T(r,H)=\{g\in G:\ g^{-1}rg\in H\}$. For $t\in T(k)$, we set $N(t,S,T)=\{g\in G(k):\ t\in gS(k)g^{-1}\text{ and }gSg^{-1}\subset T\}$. Then $T(t,H)=G_t(k)N(t,S,T)H(k)$. Let $\mathcal{J}(T,S)$ be the set of subtori of $T$ that are $G(k)$-conjugates of $S$. Since the set $\mathcal{J}(T,S)$ is independent of choices of $S$, we denote $\mathcal{J}(T):=\mathcal{J}(T,S)$. For any nonempty subset $\jmath \subset \mathcal{J}(T)$, let
$$
T_\jmath=\bigcap_{K\in\jmath}K-\bigcup_{R\in\mathcal{J}(T)-\jmath}R. 
$$
If $\jmath=\emptyset$, we set $T_\jmath=T-\bigcup_{K\in\mathcal{J}(T)}K$. For $\jmath \in 2^{\mathcal{J}(T)}$ and $\iota \in I(T)$, let $T_{\jmath,\iota}=\{t\in T_\jmath:\ G_t=G_\iota\}$. Then for any $t_1,t_2\in T_{\jmath,\iota}$, we have $T(t_1,H)=T(t_2,H)$. We denote the above set by $T(\jmath,\iota)$. The Frobenius $F$ acts on $T$, hence on $I\left(T\right)$ and $\mathcal{J}(T)$. We have
the following refinement 
\[
\frac{1}{\left|G^{F}\right|}\underset{g\in G^{F}}{\sum}f\left(g\right)=\underset{\stackrel{\jmath \in 2^{\mathcal{J}(T),F}}{\iota\in I\left(T\right)^{F}}}{\sum}\,\underset{\left[u\right]\in\mathcal{U}\left(G_{\iota}^{F}\right)}{\sum}\,\underset{s\in T_{\jmath,\iota}^{F}}{\sum}\frac{1}{\left|\bar{N}_{G}\left(\iota,T\right)^{F}\right|\left|C_{G_{\iota}}\left(u\right)^{F}\right|}f\left(su\right),
\]
where $\bar{N}_{G}\left(\iota,T\right)^{F}=\bar{N}_{G}\left(s,T\right)^{F}$
for any $s\in T_{\iota}^{F}$.

\subsection{Functions of geometric type}

We introduce the notion of functions of geometric type.
\begin{defn}
Let $\mathcal{P}\subset\mathbb{Z}^{+}$ be an arithmetic progression.
A function $M:\mathcal{P}\rightarrow\mathbb{C}$ is called of geometric
type if it is of the form 
\[
M\left(\nu\right)=\frac{\sum_{i=1}^ka_{i}\alpha_{i}^{\nu}}{\sum_{j=1}^l b_{j}\beta_{j}^{\nu}},
\]
for any $\nu\in\mathcal{P}$, where $a_{i},\alpha_{i},b_{j},\beta_{j}\in\mathbb{C}$
and the denominator is nonzero for any $\nu\in\mathcal{P}$.
\end{defn}
We recall the following lemma in \cite{LMS} for later use.
\begin{lem}\label{lem2.2}
Let $M$ be a function of geometric type defined on an arithmetic
progression $\mathcal{P}$. If $M$ is integer valued and has a finite
limit as $\nu\rightarrow\infty$ along $\mathcal{P}$, then $M$ is
a constant function.
\end{lem}

We fix a prime $\ell\neq p$ and identify $\bar{\mathbb{Q}}_{\ell}\cong\mathbb{C}.$
For a quasi-projective scheme $X$ defined over $k$,
denote by $D^{b}\left(X,\bar{\mathbb{Q}_{\ell}}\right)$ the bounded
derived category of constructible $\ell$-adic sheaves on $X$. For
$\mathcal{F}\in D^{b}\left(X,\bar{\mathbb{Q}}_{\ell}\right)$, the
Grothendieck-Lefschetz trace formula relates the local and global
traces of the geometric Frobenius $F$
\[
\underset{x\in X^{F^{\nu}}}{\sum}\underset{i}{\sum}\left(-1\right)^{i}\text{Tr}\left(F_{x}^{\nu},H^{i}\left(\mathcal{F}_{x}\right)\right)=\underset{i}{\sum}\left(-1\right)^{i}\text{Tr}\left(F^{\nu},H_{c}^{i}\left(X,\mathcal{F}_{\bar{k}}\right)\right),\ \text{for }\nu\in\mathbb{Z}^{+}.
\]
The right hand side of the above equality can be viewed as a function
of $\nu\in\mathbb{Z}^{+}$, which is of geometric type. The inner
sum on the left hand side is a function on $X^{F^{\nu}}$ associated
to $\mathcal{F}$. This procedure is called Grothendieck's sheaf-function
correspondence: 
\[
f^{\mathcal{F},\left(\nu\right)}\left(x\right)=\underset{i}{\sum}\left(-1\right)^{i}\text{Tr}\left(F_{x}^{\nu},H^{i}\left(\mathcal{F}_{x}\right)\right).
\]

\subsection{Weil representations and character formulae.}\label{sec2.4}

Assume $|k|$ is odd. Let $G$ be either $\text{Sp}_{2n}$ or $\text{U}_{n}$
defined over $k$ with Frobenius $F$. Let $V$ be a
symplectic $\bar{k}$-space of dimension $2n$ defined
over $k$. We have an action of $G$ to $V$, thus induces an embedding $G\hookrightarrow\text{Sp}\left(V\right)$. We fix
a nontrivial additive character $\psi$ of $k$. Following
\cite{Howe}, we have the Weil representation $\omega_{\psi}$ of
$G^{F}$.

Assume for the moment that $G=\text{Sp}\left(V\right)$. For any $\nu\in\mathbb{Z}^{+}$,
we have the Weil representation $\omega_{\psi}^{\left(\nu\right)}$
of $G^{F^{\nu}}$ associated to the character $\psi^{\left(\nu\right)}=\psi\circ\text{Tr}_{k_\nu/k}$,
where $k_\nu$ is the degree $\nu$ extension of $k$. Noting that $\dim\omega_{\psi}^{\left(\nu\right)}=|k|^{\nu n}$. We recall
the main result in \cite{GH}, which gives us a geometrization of the Weil representation.
\begin{thm}\label{thm2.3}
There exists a geometrically irreducible sheaf $\mathcal{F}\in D^{b}\left(G\times V\right)$
of pure weight $2n$ such that $\mathcal{F}\left[\dim G\right]$ is
perverse and 
\[
f^{\mathcal{F},\left(\nu\right)}\left(g\right)=\omega_{\psi}^{\left(\nu\right)}\left(g\right),\ \ \ g\in G^{F^{\nu}},
\]
note that here we consider the restriction of $\mathcal{F}$ to $G$.
\end{thm}

A consequence of Theorem 2.3 is that certain sums of character values
involving $\omega_{\psi}^{\left(\nu\right)}$ are functions of geometric
type by the Grothendieck-Lefchetz trace formula. For later uses, we
now recall the character formula of $\omega_{\psi}$ for semisimple
elements. Let $T$ be an $F$-stable maximal torus of $G$. Then 
\[
T^{F}\cong\underset{j}{\prod}\left(k_j^\times\right)^{\lambda_{j}}\times\left(k_{2j}^1\right)^{\lambda_{j}^{\prime}},
\]
where $\underset{j}{\sum}j\left(\lambda_{j}+\lambda_{j}^{\prime}\right)=n$. Here $k_j$ is the degree-$j$ extension of $k$ and $k_{2j}^1$ is the subgroup of norm-1 elements of $k_{2j}^\times$ (with respect to $k_j$).
We let $\left|\lambda\right|=\underset{j}{\sum}j\lambda_{j}$ and
$\left|\lambda^{\prime}\right|=\underset{j}{\sum}j\lambda_{j}^{\prime}$.
Let $\vartheta_{j}$ and $\vartheta_{j}^{\prime}$ be the unique quadratic
characters of $k_j^{\times}$ and $k_{2j}^{1}$. Let $\vartheta_{T}$ be the quadratic character of $T^{F}$ given by the product of the $\vartheta_{j}$'s and $\vartheta_{j}^{\prime}$'s under the above isomorphism. For $s\in T^{F}$, we write 
\[
s=\left(s_{j1},\ldots,s_{j\lambda_{j}};s_{j1}^{\prime},\ldots,s^\prime_{j\lambda_{j}^{\prime}}\right).
\]
By \cite[Corollary 4.8.1]{Ger}, it follows that 
\[
\omega_{\psi}\left(s\right)=\left(-1\right)^{l\left(T^{F},s\right)}\vartheta_{T}\left(s\right)q^{\frac{1}{2}\dim V^{s}},
\]
where $V^{s}=\ker\left(s-\text{id}_{V}\right)$ and $l\left(T^{F},s\right)=\left|\left\{ \left(j,k\right):1\leq k\leq\lambda_{j}^{\prime},\ s_{jk}^{\prime}\neq1\right\} \right|$.
\begin{rem}\label{rem2.4}
For $G=\text{U}_{n}$, then $\lambda_{j}=0$ for $j$ odd, and $\lambda_{j}^{\prime}=0$
for $j$ even. In this case, we have 
\[
\omega_{\psi}\left(s\right)=\left(-1\right)^{n}\vartheta_{T}\left(s\right)\left(-q\right)^{\frac{1}{2}\dim V^{s}},\ \ \ s\in T^{F}.
\]
\end{rem}

\section{A multiplicity formula for the twisted Gan-Gross-Prasad problem over finite fields}\label{sec3}

Assume $H$ is either $\text{U}_{n}$ or $\text{Sp}_{2n}$ defined
over $k$. Let $q=|k|$. We denote by $k^\prime$ be the quadratic extension of $k$. Let $G=\text{Res}_{k^\prime/k}\left(H_{k^\prime}\right)$,
where $H_{k^\prime}$ is the base change of $H$ to $k^\prime$
and $\text{Res}_{k^\prime/k}$ is the Weil
restriction of coefficients. Let $\psi$ be a nontrivial additive
character of $k$ and $\omega_{\psi}$ be the Weil representation
of $H^{F}$. We want to determine $\dim\text{Hom}_{H^{F}}\left(\pi,\omega_{\psi}\right)$, where $\pi$ is an irreducible representation of $G^{F}$. We follow
the approach in \cite{LMS}. We derive a formula for the multiplicity
\[
M\left(1\right)=\left\langle R_{T,\chi}^{G},\omega_{\psi}\right\rangle _{H^{F}}=\left\langle R_{T,\chi}^{G},\text{Ind}^{G^F}_{H^F}(\omega_{\psi})\right\rangle _{G^{F}},
\]
associated to a Deligne-Lusztig character $R_{T,\chi}^{G}$, where
$T\in\mathcal{T}\left(G\right)$ and $\chi$ is a character of $T^{F}$.

We want to show the integer valued function
\[
M\left(\nu\right)=\left\langle R_{T,\chi^{\left(\nu\right)}}^{G},\omega_{\psi}^{\left(\nu\right)}\right\rangle _{H^{F^{\nu}}}=\left\langle R_{T,\chi^{\left(\nu\right)}}^{G},\text{Ind}^{G^{F^\nu}}_{H^{F^\nu}}(\omega_{\psi}^{\left(\nu\right)})\right\rangle _{G^{F^{\nu}}}
\]
is of geometric type and has a finite limit as $\nu\rightarrow\infty$
along some arithmetic progression $\mathcal{P}_{m}=\left\{ 1+mz:\,z\in\mathbb{Z}^{+}\right\} $
starting from $\nu=1$, where $\chi^{\left(\nu\right)}=\chi\circ N_{\nu}^{T}$
is a character of $T^{F^{\nu}}$. This gives us $M$ is a constant
on $\mathcal{P}_{m}$ and thus 
\[
M\left(1\right)=\underset{\nu\rightarrow\infty,\nu\in\mathcal{P}_{m}}{\lim}M\left(\nu\right).
\]
We then give an explicit formula for regular Deligne-Lusztig characters
based on the above result.

\subsection{Multiplicity as a function of geometric type.}

We prove the first step in our approach.
\begin{prop}
There exists $m\in\mathbb{Z}^{+}$ such that $M\left(\nu\right)$
defines a function of geometric type on $\mathcal{P}_{m}$.
\end{prop}
\begin{proof}
Let 
\[
f^{\left(\nu\right)}:G^{F^{\nu}}\rightarrow\mathbb{C},\ \ f^{\left(\nu\right)}\left(g\right)=\overline{R_{T,\chi^{\left(\nu\right)}}^{G}\left(g\right)}\cdot\left[\text{Ind}^{G^{F^\nu}}_{H^{F^\nu}}(\omega_{\psi})\right]\left(g\right).
\]
It follows that 
\[
M\left(\nu\right)=\frac{1}{\left|G^{F^{\nu}}\right|}\underset{g\in G^{F^{\nu}}}{\sum}f^{\left(\nu\right)}\left(g\right).
\]
For $g=su\in G^F$ such that $s\in T_{\jmath,\iota}$ (here we use the Jordan decomposition), we have  
$$
T(g,H)=T(s,H)\cap T(u,H)=T(\jmath,\iota)\cap T(u,H)=:T(\jmath,\iota,u).
$$
By section \ref{sec2.2}, it follows that
$$
M(\nu)=\sum_{\stackrel{\jmath\in 2^{\mathcal{J}(T),F^\nu}}{\iota \in I(T)^{F^\nu}}}\sum_{\bar{\gamma}\in\bar{N}(\iota,T)^{F^\nu}}\sum_{[u]\in\mathcal{U}(G_\iota^{F^\nu})} \frac{Q^{G_\iota,\nu}_{\gamma T \gamma^{-1}}(u)}{|\bar{N}(\iota,T)^{F^\nu}|\cdot |C_{G_\iota}(u)^{F^\nu}|}$$
$$
\left[\sum_{s\in T_{\jmath,\iota}^{F^\nu}}\chi\circ N^\nu(\gamma^{-1}s\gamma)\cdot\left(\sum_{x\in T(\jmath,\iota,u)}\frac{\omega_\psi(x^{-1}sux)}{|H^{F^\nu}|}\right)\right].
$$
We can choose $m$ such that $I(T)^{F^\nu}=I(T)^F$, $\mathcal{J}(T)^{F^\nu}=\mathcal{J}(T)^{F}$, $\bar{N}(\iota,T)^{F^\nu}=\bar{N}(\iota,T)^{F}$ and $\mathcal{U}(G_\iota^{F^\nu})=\mathcal{U}(G_\iota^{F})$, for any $\nu \in \mathcal{P}_m$. In this case 
$$
M(\nu)=\sum_{\stackrel{\jmath\in 2^{\mathcal{J}(T),F}}{\iota \in I(T)^{F}}}\sum_{\bar{\gamma}\in\bar{N}(\iota,T)^{F}}\sum_{[u]\in\mathcal{U}(G_\iota^{F})} \frac{Q^{G_\iota,\nu}_{\gamma T \gamma^{-1}}(u)}{|\bar{N}(\iota,T)^{F}|\cdot |C_{G_\iota}(u)^{F^\nu}|}$$
$$
\left[\sum_{s\in T_{\jmath,\iota}^{F^\nu}}\chi\circ N^\nu(\gamma^{-1}s\gamma)\cdot\left(\sum_{x\in T(\jmath,\iota,u)}\frac{\omega_\psi(x^{-1}sux)}{|H^{F^\nu}|}\right)\right].
$$
Since $Q^{G_\iota,\nu}_{\gamma T \gamma^{-1}}(u)$ and $|C_{G_\iota}(u)^{F^\nu}|$ are polynomials (with respect to $\nu$), it suffices to show that 
$$
\sum_{s\in T_{\jmath,\iota}^{F^\nu}}\chi\circ N^\nu(\gamma^{-1}s\gamma)\cdot\left(\sum_{x\in T(\jmath,\iota,u)}\frac{\omega_\psi(x^{-1}sux)}{|H^{F^\nu}|}\right)
$$
is of geometric type. As in Theorem \ref{thm2.3}, we take $\mathcal{F}\in D^b(H\times V)$ to be the sheaf corresponding to $\omega_\psi$. Let $
i:\frac{H}{H}\rightarrow \frac{G}{G}$ be the natural morphism mapping conjugacy classes of $H$ to $G$. Then $i_!\mathcal{F}$ is the corresponding sheaf of $\text{Ind}^G_H(\omega_\psi)$. We denote by $\mathcal{L}_{^\gamma\chi}$ the rank one sheaf on $T_{\jmath,\iota}$ corresponding to the character $^\gamma\chi$. We have 
$$
\sum_{s\in T_{\jmath,\iota}^{F^\nu}}\chi\circ N^\nu(\gamma^{-1}s\gamma)\cdot\left(\sum_{x\in T(\jmath,\iota,u)}\frac{\omega_\psi(x^{-1}sux)}{|H^{F^\nu}|}\right)=\text{Tr}(F^\nu,R\Gamma_c(T_{\jmath,\iota},u^*i_!\mathcal{F}\otimes\mathcal{L}^\vee_{^\gamma\chi})),
$$
which is to say it is of geometric type. Therefore, the function $M(\nu)$ is of geometric type for $\nu\in\mathcal{P}_m$.
\end{proof}
Now there remains to show $M\left(\nu\right)$ converges when $\nu\overset{\mathcal{P}_{m}}{\longrightarrow}\infty$
and compute the limit explicitly.

\subsection{A partition of $T_{\jmath,\iota}$ via its action on the symplectic space $V$.}

We mimic the set up in \cite{LMS}. Since we are considering $H$
is either $\text{U}_{n}$ or $\text{Sp}_{2n}$, it follows that $H$
acts on the $2n$-dimensional symplectic space $V$. Define $\Gamma=\text{Gal}\left(\bar{k}/k\right)\times\left\langle \imath\right\rangle$, where $\imath$ is the involution $x\mapsto x^{-1}:\bar{k}\rightarrow\bar{k}$.
We have a group action of $\Gamma$ on $\bar{k}^{\times}$.
Let $\Lambda_{s}\subset\bar{k}^{\times}$ be the set
of eigenvalues of $s$ acting on $V$. Then $\Gamma$ acts on $\Lambda_{s}$,
and we denote the $\Gamma$-orbit of $a\in\Lambda_{s}$ by $\left[a\right]$. Observe 
\[
H_{s}=\underset{\left[a\right]\in\Lambda_{s}/\Gamma}{\prod}H_{s,\left[a\right]},
\]
where 
\begin{itemize}
\item $H_{s,\left[1\right]}$ and $H_{s,\left[-1\right]}$ are of the same
type as $H$;
\item $H_{s,\left[a\right]}$, with $a\neq\pm1$, are general linear groups
or unitary groups.
\end{itemize}
Let 
\[
H_{s}^{\prime}=\underset{\left[a\right]\neq\left[1\right]}{\prod}H_{s,\left[a\right]}.
\]
We have $H_{s}=H_{s}^{\prime}\times H_{s,\left[1\right]}$. We fix an $F$-stable maximal torus $S$ of $H$. Let $K_H(S)=\{V^s:\ s\in S\}$ be an index set of symplectic subspaces of $V$. For any $\kappa\in K_H(S)$, we denote by $S_\kappa=\{s\in S:\ V^s=V^\kappa\}$ and $H_\kappa$ the centralizer of $S_\kappa$. We have
$$
H_\kappa=Z_\kappa\times H_{\kappa,[1]},
$$
where $Z_\kappa=\overline{S_\kappa}$ can be viewed as a maximal torus in $\text{Sp}(V/V^\kappa)$ and $H_{\kappa,[1]}=H_{s,[1]}$, for any $s\in S_\kappa$. For each $(\jmath,\iota)\in 2^{\mathcal{J}(T)}\times I(T)$, let $x\in G(k)$ such that $x^{-1}T_{\jmath,\iota}x\subset S$ and denote by $K_{H}\left(T_{\jmath,\iota}\right)$
be the index set for $\left\{ V^{x^{-1}sx}:\,s\in T_{\jmath,\iota}\right\}$. Then
$\left|K_{H}\left(T_{\jmath,\iota}\right)\right|\leq 2^{n}$. The Frobenius $F$ acts
naturally on $K_{H}\left(T_{\jmath,\iota}\right)$ and 
\[
K_{H}\left(T_{\jmath,\iota}\right)^{F}=\left\{ V^{s}:\,s\in T_{\jmath,\iota}^{F}\right\} .
\]
For $\jmath\in K_{H}\left(T\right)$, we denote by $V^{\kappa}$ the
corresponding symplectic subspace of $V$ and 
\[
T_{\jmath,\iota,\kappa}=\left\{ s\in T_{\jmath,\iota}:\,V^{x^{-1}sx}=V^{\kappa}\right\}.
\]
Then 
\[
T_{\jmath,\iota}=\underset{\kappa\in K_{H}\left(T_{\jmath,\iota}\right)}{\bigsqcup}T_{\jmath,\iota,\kappa}.
\]
We increase the divisibility of $m$ so that $F^{m}$ acts trivially
on $K_{H}\left(T_{\jmath,\iota}\right)$, for any $(\jmath,\iota)\in 2^{\mathcal{J}(T),F}\times I(T)^F$. If so, then for any $\nu\in\mathcal{P}_{m}$, we have
\[
T_{\jmath,\iota}^{F^{\nu}}=\underset{\kappa\in K_{H}\left(T_{\jmath,\iota}\right)^{F}}{\bigsqcup}T_{\jmath,\iota,\kappa}^{F^{\nu}}.
\]
For any $\kappa\in K_{H}\left(T_{\jmath,\iota}\right)^{F}$, we have the following decomposition
\[
G_{\iota}=G_{\iota,\kappa}^{\prime}\times G_{\kappa,\left[1\right]},
\]
where $G_{\iota,\kappa}^{\prime}=x^{-1}(\text{Res}_{k^\prime/k}\,H_{x^{-1}sx,k^\prime}^{\prime})x$ and $G_{\kappa,\left[1\right]}=x^{-1}(\text{Res}_{k^\prime/k}\,H_{x^{-1}sx,[1],k^\prime})x$ for any $s\in T_{\jmath,\iota,\kappa}$, noting that this does not depend on $\jmath\in 2^{\mathcal{J}(T),F}$. We denote $Z_{\iota,\kappa}=Z(G^\prime_{\iota,\kappa})$.

\subsection{Multiplicity formula.}

Recall that we have
$$
M(\nu)=\sum_{\stackrel{\jmath\in 2^{\mathcal{J}(T),F}}{\iota \in I(T)^{F}}}\sum_{\bar{\gamma}\in\bar{N}(\iota,T)^{F}}\sum_{[u]\in\mathcal{U}(G_\iota^{F})}\sum_{\kappa\in K_H(T_{\jmath,\iota})^F} \frac{Q^{G_\iota,\nu}_{\gamma T \gamma^{-1}}(u)}{|\bar{N}(\iota,T)^{F}|\cdot |C_{G_\iota}(u)^{F^\nu}|}$$
$$
\left[\sum_{s\in T_{\jmath,\iota,\kappa}^{F^\nu}}\chi\circ N^\nu(\gamma^{-1}s\gamma)\cdot\left(\sum_{x\in T(\jmath,\iota,u)}\frac{\omega_\psi(x^{-1}sux)}{|H^{F^\nu}|}\right)\right].
$$
Observe
$$
\deg(Q^{G_\iota,\nu}_{\gamma T \gamma^{-1}}(u))= \frac{1}{2}(\dim C_{G_\iota}(u)-\overline{\text{rk}}\,G_\iota)\text{ and }\dim T(\jmath,\iota,u)\leq \frac{1}{2}\dim C_{G_\iota}(u)+\dim H. 
$$
By Remark \ref{rem2.4}, we have 
\[
\left|\omega_{\psi}^{\left(\nu\right)}\left(x^{-1}sux\right)\right|=q^{\frac{\nu}{2}\dim V^{x^{-1}sux}}\leq q^{\frac{\nu}{2}\dim V^{x^{-1}sx}}=q^{\frac{\nu}{2}\overline{\text{rk}}G_{\kappa,\left[1\right]}}.
\]
This gives us 
$$
\left|\sum_{s\in T_{\jmath,\iota,\kappa}^{F^\nu}}\chi\circ N^\nu(\gamma^{-1}s\gamma)\cdot\left(\sum_{x\in T(\jmath,\iota,u)}\frac{\omega_\psi(x^{-1}sux)}{|H^{F^\nu}|}\right)\right|
\leq \left|T_{\jmath,\iota,\kappa}^{F^{\nu}}\right|\cdot q^{\frac{\nu}{2}\overline{\text{rk}}G_{\kappa,\left[1\right]}}
$$
$$
\ll_{v,\mathcal{P}_{m}} q^{\frac{\nu}{2} \left(\overline{\text{rk}}Z_{\iota,\kappa}+\overline{\text{rk}}G_{\kappa,\left[1\right]}\right)}
\leq q^{\frac{\nu}{2} \left(\overline{\text{rk}}G_{\iota,\kappa}+\overline{\text{rk}}G_{\kappa,\left[1\right]}\right)}= q^{\frac{\nu}{2}\overline{\text{rk}}G_{\iota}}=q^{\nu n}.
$$
Therefore 
\[
\frac{Q^{G_\iota,\nu}_{\gamma T \gamma^{-1}}(u)}{|\bar{N}(\iota,T)^{F}|\cdot |C_{G_\iota}(u)^{F^\nu}|}\left[\sum_{s\in T_{\jmath,\iota,\kappa}^{F^\nu}}\chi\circ N^\nu(\gamma^{-1}s\gamma)\cdot\left(\sum_{x\in T(\jmath,\iota,u)}\frac{\omega_\psi(x^{-1}sux)}{|H^{F^\nu}|}\right)\right]\rightarrow 0
\]
when $\nu\overset{\mathcal{P}_{m}}{\longrightarrow}\infty$ unless 
\[
u=1\ \ \text{ and }\ \ \overline{\text{rk}}\left(T_{\jmath,\iota,\kappa}\right)=\frac{1}{2}\overline{\text{rk}}Z_{\iota,\kappa}=\frac{1}{2}\overline{\text{rk}}\ G_{\iota,\kappa}^{\prime}.
\]
The second equality is equivalent to $G^\prime_{\iota,\kappa}=Z_{\iota,\kappa}=\overline{T_\iota}=\text{Res}_{k^\prime/k}(\overline{T_{\jmath,\iota,\kappa}})_{k^\prime}$. In this case, it follows that $T_{\jmath,\iota,\kappa}$ is $G(k)$-conjugate to $Z_{\kappa,reg}$. We denote by $\chi_{\jmath,\kappa}$ as a character of $Z_{\kappa,reg}$ given by conjugating $\chi\mid_{T_{\jmath,\iota,\kappa}}$ with an element $x\in G(k)$ such that $x^{-1}T_{\jmath,\iota,\kappa}x=Z_{\kappa,reg}$. By the character formula stated in Section \ref{sec2.4}, we have 
\[
\omega_{\psi}^{\left(\nu\right)}\left(s\right)=\left(-1\right)^{l\left(Z_{\kappa}^{F^{\nu}},s\right)}\vartheta_{Z_{\kappa}}^{\left(\nu\right)}\left(s\right)q^{\frac{\nu}{2}\dim V^{\kappa}},
\]
for any $s\in Z_{\kappa}^{\text{reg},F^{\nu}}$. We can choose $m\in\mathbb{Z}^{+}$ such that $l\left(Z_{\kappa}^{F^{\nu}},s\right)=l\left(Z_{\kappa}^{F},s\right)$
for all $\nu\in\mathcal{P}_{m}$. We can see that $l\left(Z_{\kappa}^{F},s\right)$
is the number of anisotropic factors of $Z_{\kappa}$, hence does not depend on choice of $s\in Z_{\kappa}^{\text{reg},F}$. We denote this number by $l\left(Z_{\kappa}\right)$. Since regular elements are Zariski open dense, in the case considered above, when $\nu \overset{\mathcal{P}_m}{\longrightarrow}\infty$, we have 
$$
\sum_{s\in T_{\jmath,\iota,\kappa}^{F^\nu}}\chi^{(\nu)}(\gamma^{-1}s\gamma)\cdot\left(\sum_{x\in T(\jmath,\iota)}\frac{\omega_\psi(x^{-1}sx)}{|H^{F^\nu}|}\right) 
$$
$$
\approx_{\nu,\mathcal{P}_m} 
(-1)^{l(Z_\kappa)}q^{\nu\cdot\overline{rk}\,H_{\kappa,[1]}}\frac{|Z_\kappa^{F^\nu}|\cdot|T(\jmath,\iota)^{F^\nu}|\cdot \langle ^\gamma\chi_\kappa^{(\nu)},\vartheta_{Z_\kappa}^{(\nu)}\rangle_{Z_\kappa^{F^\nu}}}{|H^{F^\nu}|}
$$
$$
\approx_{\nu,\mathcal{P}_m} (-1)^{l(Z_\kappa)}q^{\nu (n+\frac{1}{2}\dim{G_\iota})}\mathcal{C}(T(\jmath,\iota))\langle ^\gamma\chi_\kappa,\vartheta_{Z_\kappa}\rangle_{Z_\kappa^{F}},
$$
where $\mathcal{C}(T(\jmath,\iota))$ is the number of connected components of $\mathcal{C}(T(\jmath,\iota))$. Moreover, since we are considering the case when $G=\text{Res}_{k^\prime,k}H_{k^\prime}$, we always have $\mathcal{C}(T(\jmath,\iota))=1$. We remark that although the definition of $\chi_\kappa$ depends on choices of $x\in G(k)$ such that $x^{-1}T_{\jmath,\iota,\kappa}x=Z_{\kappa,reg}$, the above formula is independent of these choices. Observe
$$
\frac{Q^{G_\iota,\nu}_{\gamma T \gamma^{-1}}(1)}{|\bar{N}(\iota,T)^{F}|\cdot |G_\iota^{F^\nu}|}\left[\sum_{s\in T_{\jmath,\iota,\kappa}^{F^\nu}}\chi\circ N^\nu(\gamma^{-1}s\gamma)\cdot\left(\sum_{x\in T(\jmath,\iota)}\frac{\omega_\psi(x^{-1}sx)}{|H^{F^\nu}|}\right)\right]
$$
converges to 
$$
\frac{(-1)^{\text{rk}_{k^\prime}(T)+\text{rk}_{k^\prime}(G)+l(Z_\kappa)}}
{|\bar{N}(\iota,T)^F|}\cdot 
\langle ^\gamma\chi_\kappa,\vartheta_{Z_\kappa}\rangle_{Z_\kappa^{F}}.
$$
From the above discussions, we can see that $\underset{\nu\overset{\mathcal{P}_m}{\longrightarrow}\infty}{\lim}M(\nu)$ is finite, which is to say
\[
M\left(1\right)=\underset{\nu\overset{\mathcal{P}_m}{\longrightarrow}\infty}{\lim}M(\nu).
\]
We want to rearrange nonzero terms in $\underset{\nu\overset{\mathcal{P}_m}{\longrightarrow}\infty}{\lim}M(\nu)$. We fix an $F$-stable maximal torus of $S$. Let $\bar{K}_H(S)=\{V^s:\ s\in S/G(k)-\text{conj}\}$. For each $\kappa \in \bar{K}_H(S)$, let $H_\kappa$ and $G_\kappa$ be the centralizer of $S_\kappa$ in $H$ and $G$. Recall that $H_\kappa=Z_\kappa \times H_{\kappa,[1]}$. For $(\jmath,\iota,\kappa)$ whose corresponding limit is nonzero, it follows that $\overline{T_\jmath}=\overline{T_{\jmath,\iota,\kappa}}$ is $G(k)$-conjugate to $Z_\kappa$ and $G_\iota$ is $G(k)$-conjugate to $G_\kappa$. For each $\kappa \in \bar{K}_H(S)$, we set $\mathcal{J}(T,Z_\kappa)$ to be the set of subtori of $T$ that are $G(k)$-conjugate to $Z_\kappa$. For each $\jmath\in \mathcal{J}(T,Z_\kappa)$, we denote $\bar{N}_G(\jmath,T)=\bar{N}_G(s,T)$ for some $s\in T_\jmath$ such that $G_s$ is $G(k)$-conjugate to $G_\kappa$. From the above discussion, we have  
$$
M(1)=\sum_{\kappa\in \bar{K}_H(S)^F}\sum_{\jmath\in \mathcal{J}(T,Z_\kappa)^F}\sum_{\gamma\in \bar{N}(\jmath,T)^F}
\frac{(-1)^{\text{rk}_{k^\prime}(T)+\text{rk}_{k^\prime}(G)+l(Z_\kappa)}}
{|\bar{N}(\jmath,T)^F|}\cdot \langle ^\gamma\chi_\jmath,\vartheta_{T_\jmath}\rangle_{T_\jmath^{F}},
$$
here we denote by $T(\jmath,\iota)$ the set $\{g\in G:\,g^{-1}sg\in H\}$, for an element $s\in T_{\jmath}$ such that $G_s$ is $G(k)$-conjugate to $G_\kappa$. We are now able to deduce the multiplicity formula.
\begin{thm}
We fix an $F$-stable torus $S$ of $H$. We have the following multiplicity 
\[
\left\langle R_{T,\chi}^{G},\omega_{\psi}\right\rangle _{H^{F}}=
\sum_{\kappa\in \bar{K}_H(S)^F}\sum_{\jmath\in \mathcal{J}(T,Z_\kappa)^F}\sum_{\gamma\in \bar{N}(\jmath,T)^F}
\frac{(-1)^{\text{rk}_{k^\prime}(T)+\text{rk}_{k^\prime}(G)+l(Z_\kappa)}}
{|\bar{N}(\jmath,T)^F|}\cdot \langle ^\gamma\chi_\jmath,\vartheta_{T_\jmath}\rangle_{T_\jmath^{F}}
\]
\[
=\sum_{\kappa\in \bar{K}_H(S)^F}\sum_{\jmath\in \mathcal{J}(T,Z_\kappa)^F}
\frac{(-1)^{\text{rk}_{k^\prime}(T)+\text{rk}_{k^\prime}(G)+l(Z_\kappa)}}
{|W_G(T)^F|}\sum_{\gamma \in W_G(T)^F}\langle ^\gamma\chi_\jmath,\vartheta_{T_\jmath}\rangle_{T_\jmath^{F}}.
\]
\end{thm}

\section{Application of the multiplicity formula}\label{sec4}
In this section, for simplicity, we only consider the case when $(G,H)=\left(\text{Res}_{k^\prime/k}\text{GL}_{n},\text{U}_{n}\right)$. The remaining case when $(G,H)=\left(\text{Res}_{k^\prime/k}\text{Sp}_{2n,k^\prime},\text{Sp}_{2n}\right)$ should be treated similarly.

\subsection{A variant of Hiss-Schroer conjecture}

In this subsection, similar to \cite{LMS} and \cite{Shi1}, we consider
a variant of \cite{HS} for the twisted Gan-Gross-Prasad problem.
\begin{problem}
Let $\pi\in\text{Irr}\left(G^{F}\right)$, where $H=\text{U}_{n}$
or $\text{Sp}_{2n}$ defined over $k$ and $G=\text{Res}_{k^\prime/k}H_{k^\prime}$.
Then the multiplicity $\left\langle \pi,\omega_{\psi}\right\rangle _{H^{F}}$
is bounded by a function of $n$ which is independent of $k$.
\end{problem}

We give an answer to the above question when $\left(G,H\right)=\left(\text{Res}_{k^\prime/k}\text{GL}_{n},\text{U}_{n}\right)$.
The remaining case should follow the same strategy to \cite{Shi1}.
\begin{thm}\label{thm4.2}
Assume that $\left(G,H\right)=\left(\text{Res}_{k^\prime/k}\text{GL}_{n},\text{U}_{n}\right)$. Then for any $\pi\in\text{Irr}\left(G^{F}\right)$,
we have 
\[
m\left(\pi\right)=\left\langle \pi,\omega_{\psi}\right\rangle _{H^{F}}\leq 2^{n}(n!)^{3/2}.
\]
\end{thm}

\begin{proof}
We follow the approach in \cite{LMS}. By Theorem 1.2, for any Deligne-Lusztig
virtual representation $R_{T,\chi}^{G}$, we have 
\[
\left|\langle R_{T,\chi}^{G},\omega_\psi\rangle_{H^F}\right|\leq
\sum_{\kappa\in \bar{K}_H(S)^F}| \mathcal{J}(T,Z_\kappa)^F|.
\]
Observe $|\bar{K}_H(S)^F|\leq 2^n$ and $|\mathcal{J}(T,Z_\kappa)|\leq n!$. This gives us 
$$
|m(R^G_{T,\chi})|\leq 2^{n}n!.
$$
By \cite[Theorem 3.2]{LS}, we can write $\pi$ as 
\[
\pi=R_{L,\chi,\rho}^{G}=\frac{\left(-1\right)^{\text{rk}G+\text{rk}L}}{\left|W_{L}\right|}\underset{w\in W_{L}}{\sum}\text{Tr}\left(w\tilde{F},V_{\rho}\right)\cdot R_{T_{w},\chi_{w}}^{G},
\]
where
\begin{itemize}
\item $L$ is an $F$-stable reductive connected subgroup of $G$ such that
$\overline{\text{rk}}L=\overline{\text{rk}}G$,
\item $\chi:L^{F}\rightarrow\mathbb{C}^{\times}$ is a character,
\item $\left(\rho,V_{\rho}\right)\in\text{Irr}\left(W_{L}\right)^{F}$ and
$\tilde{F}$ is a finite order automorphism of $V_{\rho}$ such that
$\tilde{F}w\tilde{F}^{-1}=F\left(w\right)$ on $V_{\rho}$ for all
$w\in W_{L}$,
\item $T_{w}\in\mathcal{T}\left(L\right)$ is of class $\left[w\right]\in H^{1}\left(F,W_{L}\right)$
and $\chi_{w}=\chi_{\mid T_{w}}$.
\end{itemize}
Observe 
\[
\left|\text{Tr}\left(w\tilde{F},V_{\rho}\right)\right|\leq\dim\rho\leq\sqrt{\left|W_{L}\right|}\leq\sqrt{n!}.
\]
Therefore 
\[
m\left(\pi\right)\leq 2^{n}(n!)^{3/2}.
\]
\end{proof}

\subsection{Unipotent case}

Since every unipotent representation of $G^{F}$ is a linear combination
of $R_{T,1}^{G}$, it suffices to compute the following multiplicity
\[
m\left(R_{T,1}^{G}\right)=\left\langle R_{T,1}^{G},\omega_{\psi}\right\rangle _{H^{F}}.
\]
By Theorem 1.1, we have 
\[
m\left(R_{T,1}^{G}\right)=\sum_{\kappa\in \bar{K}_H(S)^F}\sum_{\jmath\in \mathcal{J}(T,Z_\kappa)^F}
\frac{(-1)^{\text{rk}_{k^\prime}(T)+\text{rk}_{k^\prime}(G)+l(Z_\kappa)}}
{|W_G(T)^F|}\sum_{\gamma \in W_G(T)^F}\langle 1,\vartheta_{T_\jmath}\rangle_{T_\jmath^{F}}.
\]
If $Z_\kappa^{F}$ is nontrivial, it follows that $\langle 1,\vartheta_{T_\jmath}\rangle_{T_\jmath^{F}}=0$, for any $\jmath\in \mathcal{J}(T,Z_\kappa)$. We consider the case when $Z_\kappa^{F}$ is trivial. In this case, we have $V^{\kappa}=V$ and $\mathcal{J}(T,Z_\kappa)$ only contains the identity torus. This gives us 
\[
m\left(R_{T,1}^{G}\right)=\left(-1\right)^{\text{rk}_{k^\prime}(T)+\text{rk}_{k^\prime}(G)}.
\]
We have the following theorem.
\begin{thm}\label{thm4.3}
Let $T$ be an $F$-stable maximal torus of $G$. We have  
\[
m\left(R_{T,1}^{G}\right)=\left\langle R_{T,1}^{G},\omega_{\psi}\right\rangle _{H^{F}}=\left(-1\right)^{\text{rk}T+\text{rk}G}.
\]
\end{thm}

\subsection{Regular case.}
Let $T$ be an $F$-stable maximal torus of $G^{F}$. In this section, we assume that $\chi$ is a regular character of $T^F$. Then $(-1)^{\text{rk}_{k^\prime}T+\text{rk}_{k^\prime}G}R^G_{T,\chi}$ is an irreducible representaton of $G^F$. We can write
\[
T^{F}=\underset{j}{\prod}\left(k_{2j}^{\times}\right)^{\lambda_{j}},
\]
where $\left|\lambda\right|=\underset{j}{\sum}j\lambda_{j}=n$. We can assume that 
$$
S^F=\prod_{j\text{ even}}(k_{j}^\times)^{\lambda_j}\times\prod_{j\text{ odd}}(k^1_{2j})^{\lambda_j}.
$$
For $\kappa\in \bar{K}_{H}\left(S\right)$, we can see that $Z_{\kappa}^{F}$
is of the form 
\[
Z_{\kappa}^{F}=\prod_{j\text{ even}}(k_{j}^\times)^{\nu_j}\times\prod_{j\text{ odd}}(k^1_{2j})^{\nu_j},
\]
where $\nu_{j}\leq\lambda_{j}$ for all $j$. We can see that $\kappa$ is uniquely determined by $\nu$. For $\jmath\in \mathcal{J}(T,Z_\kappa)$, we have $T_\jmath^F$ is of the following form
$$
T_\jmath^F=\prod_{j\text{ even}}(k_{j}^\times)^{\nu_j}\times\prod_{j\text{ odd}}(k_{2j}^\times)^{\mu_j}\times(k^1_{2j})^{\mu^\prime_j},
$$
where $2\mu_{j}+\mu^\prime_j=\nu_j$ when $j$ is odd. There are 
\[
\underset{j\text{ even}}{\prod}
\left( \begin{matrix}
\lambda_{j} \\
\nu_{j}
\end{matrix} \right)
\underset{j\text{ odd}}{\prod}
\left( \begin{matrix}
\lambda_{j} \\
\mu^\prime_{j}
\end{matrix} \right)
\frac{(\lambda_j-\mu_j^\prime)!}
{2^{\mu_j}(\lambda_j-\nu_j)!\mu_j!}
\]
elements $T_\jmath^{F}\in\mathcal{J}\left(T,Z_\kappa\right)^{F}$
corresponding to $\left(\mu,\mu^{\prime}\right)$. Moreover, we have 
$$
|\bar{N}(\jmath,T)^F|=\underset{j\text{ even}}{\prod}
j^{\nu_j}\cdot\nu_j!\cdot\left( \begin{matrix}
\lambda_{j} \\
\nu_{j}
\end{matrix} \right)
\underset{j\text{ odd}}{\prod}
j^{\mu_j+\mu_j^\prime}\cdot
\mu_j!\mu_j^\prime!
\left( \begin{matrix}
\lambda_{j} \\
\mu^\prime_{j}
\end{matrix} \right)
\frac{(\lambda_j-\mu_j^\prime)!}
{(\lambda_j-\nu_j)!\mu_j!}.
$$

Theorem \ref{thm1.1} allows us to compute explicitly $\left\langle R_{T,\chi}^{G},\omega_{\psi}\right\rangle _{H^{F}}$
as long as we have a parametrization of $T$ and $\chi$. Let $\Gamma=\text{Gal}(\bar{k}/k)$ and on the $j$-th blocks of $T^F$, we write
$$
\chi_j=\chi_{j1}\otimes \ldots \otimes \chi_{j\lambda_j}.
$$
We recall the quadratic character $\vartheta_j$ of $k_{j}^\times$ when $j$ is even or $k_{2j}^1$ when $j$ is odd. We define
$$
I_j=\{t\in[1,\lambda_j]:\ \chi_{jt}|_{k_{j}^\times}\in \Gamma\cdot \vartheta_{j}\}
$$
when $j$ is even or
$$
I_j=\{\{t_1,t_2\}\subset [1,\lambda_{j}]:\ (\chi_{jt_1}\otimes\chi_{jt_2})|_{k_{2j}^\times}\in \Gamma\cdot \vartheta_{2j}\}
$$
and 
$$
I_j^\prime=\{t\in [1,\lambda_{j}]:\ \chi_{jt}|_{k_{2j}^1}\in \Gamma\cdot \vartheta_{j}\}
$$
when $j$ is odd. We fix a pair $(\mu,\mu^\prime)$. Since $\chi$ is regular, it follows that
$$
\sum_{\gamma \in \bar{N}(\jmath,T)^F}\langle ^\gamma\chi_\jmath,\vartheta_{T_\jmath}\rangle_{T_\jmath^{F}}
=
\underset{j\text{ even}}{\prod}
j^{\nu_j}
\left( \begin{matrix}
|I_j| \\
\nu_j
\end{matrix} \right)\nu_j!\cdot
\underset{j\text{ odd}}{\prod}
j^{\mu_j+\mu_j^\prime}\cdot 2^{\mu_j}\cdot
\left( \begin{matrix}
|I_j| \\
\mu_j
\end{matrix} \right)
\left( \begin{matrix}
|I_j^\prime| \\
\mu_j^\prime
\end{matrix} \right)
\mu_j!\mu_j^\prime!.
$$
By Theorem \ref{thm1.1}, we obtain
$$
\left\langle R_{T,\chi}^{G},\omega_{\psi}\right\rangle _{H^{F}}=
\sum_{\kappa\in \bar{K}_H(S)^F}\sum_{\jmath\in \mathcal{J}(T,Z_\kappa)^F}\sum_{\gamma\in \bar{N}(\jmath,T)^F}
\frac{(-1)^{\text{rk}_{k^\prime}(T)+\text{rk}_{k^\prime}(G)+l(Z_\kappa)}}
{|\bar{N}(\jmath,T)^F|}\cdot \langle ^\gamma\chi_\jmath,\vartheta_{T_\jmath}\rangle_{T_\jmath^{F}}
$$
$$
=\sum_{\nu}\sum_{(\mu,\mu^\prime)}(-1)^{\sum_j \lambda_j+\sum_j \mu_j^\prime+n}
\underset{j\text{ even}}{\prod}
\left( \begin{matrix}
|I_j| \\
\nu_j
\end{matrix} \right)
\underset{j\text{ odd}}{\prod}
\left( \begin{matrix}
|I_j| \\
\mu_j
\end{matrix} \right)
\left( \begin{matrix}
|I_j^\prime| \\
\mu_j^\prime
\end{matrix} \right)= 
\begin{cases}
2^{r} & \text{if }\left|I_{j}^{\prime}\right|=\emptyset\text{, for all }j\\
0 & \text{otherwise,}
\end{cases}
$$
where $r=\sum_j|I_j|$.
In particular, we now consider the case when $T$ is anisotropic mod $Z\left(G\right)$. This corresponds to the case when $(-1)^{\text{rk}_{k^\prime}T+\text{rk}_{k^\prime}G}R^G_{T,\chi}$ is cuspidal. In this case, observe $T^{F}\simeq k^\times_{2n}$. When $n$ is even, it follows that 
$$
\langle(-1)^{\text{rk}_{k^\prime}T+\text{rk}_{k^\prime}G}R^G_{T,\chi},\omega_\psi\rangle_{H^F}=
\begin{cases}
1 & \text{if }\chi|_{k_n^\times} \notin \Gamma \cdot \vartheta_n,
\\
2 & \text{otherwise.}
\end{cases}
$$
When $n$ is odd, we have
$$
\langle(-1)^{\text{rk}_{k^\prime}T+\text{rk}_{k^\prime}G}R^G_{T,\chi},\omega_\psi\rangle_{H^F}=
\begin{cases}
1 & \text{if }\chi|_{k^1_{2n}} \notin \Gamma \cdot \vartheta^\prime_n,
\\
0 & \text{otherwise.}
\end{cases}
$$
The above computations can be summarized to the following theorem.
\begin{thm}\label{thm4.4}
Let $T$ be an $F$-stable maximal torus of $G$ which is anisotropic
mod $Z\left(G\right)$. Then  
$$
\langle(-1)^{\text{rk}_{k^\prime}T+\text{rk}_{k^\prime}G}R^G_{T,\chi},\omega_\psi\rangle_{H^F}=
\begin{cases}
1 & \text{if }\chi|_{k_n^\times} \notin \Gamma \cdot \vartheta_n,
\\
2 & \text{otherwise}
\end{cases}
$$
when $n$ is even and
$$
\langle(-1)^{\text{rk}_{k^\prime}T+\text{rk}_{k^\prime}G}R^G_{T,\chi},\omega_\psi\rangle_{H^F}=
\begin{cases}
1 & \text{if }\chi|_{k^1_{2n}} \notin \Gamma \cdot \vartheta^\prime_n,
\\
0 & \text{otherwise}
\end{cases}
$$
when $n$ is odd.
\end{thm}


\begin{thebibliography}{Wang1}
\bibitem[Ato]{Ato}Atobe, H. (2018). The local theta correspondence and the local Gan–Gross–Prasad conjecture for the symplectic-metaplectic case. Mathematische Annalen, 371(1), 225-295.

\bibitem[BP1]{BP1}Beuzart-Plessis, R. (2016). La conjecture locale de Gross-Prasad pour les repr\'esentations temp\'er\'ees des groupes unitaires. Mém. Soc.
Math. Fr. (N.S.), no. 149, vii+191 pp.

\bibitem[BP2]{BP2}Beuzart-Plessis, R. (2015). Endoscopie et conjecture locale raffinée de Gan–Gross–Prasad pour les groupes unitaires. Compositio Mathematica, 151(7), 1309-1371.

\bibitem[DL]{DL}Deligne, P., Lusztig, G. (1976). Representations of reductive groups over finite fields. Annals of Mathematics, 103(1), 103-161.

\bibitem[Ger]{Ger}Gérardin, P. (1977). Weil representations associated to finite fields. Journal of Algebra, 46(1), 54-101.

\bibitem[GH]{GH}Gurevich, S., Hadani, R. (2008). The geometric Weil representation. Selecta Mathematica, 13(3), 465.

\bibitem[GP1]{GP1}Gross, B. H., Prasad, D. (1992). On the decomposition of a representation of $\text{SO}_{n}$ when restricted to $\text{SO}_{n-1}$. Canadian Journal of Mathematics, 44(5), 974-1002.

\bibitem[GP2]{GP2}Gross, B. H., Prasad, D. (1994). On irreducible representations
of $\text{SO}_{2n+1}\times\text{SO}_{2m}$, Canadian Journal of Mathematics 46,
930--950.

\bibitem[GGP1]{GGP1}Gan, W. T., Gross, B. H., Prasad, D. (2012). Symplectic local root numbers, central critical L values, and restriction problems in the representation theory of classical groups. Astérisque, 346(1), 1-109.

\bibitem[GGP2]{GGP2}Gan, W. T., Gross, B. H., Prasad, D. (2012). Restrictions of representations of classical groups: examples. Astérisque, 346(2), 111--170.

\bibitem[GGP3]{GGP3}Gan, W. T., Gross, B. H., Prasad, D. (2023). Twisted GGP problems and conjectures. Compositio Mathematica, 159(9), 1916-1973.

\bibitem[GI]{GI}Gan, W. T., Ichino, A. (2016). The Gross–Prasad conjecture and local theta correspondence. Inventiones mathematicae, 206(3), 705-799.

\bibitem[Howe]{Howe}Howe R. (1973). Invariant theory and duality for classical
groups over finite fields with applications to their singular representation
theory, unpublished paper, available at https://blog.nus.edu.sg/matzhucb/links/.

\bibitem[HS]{HS}Hiss, G., Schröer, M. (2020). Two conjectures on the Weil representations of finite symplectic and unitary groups. Journal of Algebra, 558, 485-490.

\bibitem[Le1]{Le1}Le, N. H. (2025). The local twisted Gan-Gross-Prasad conjecture for $\text {GL (V)/U (V)} $. arXiv preprint arXiv:2504.21002.

\bibitem[Le2]{Le2}Le, N. H. (2025). The local twisted Gan-Gross-Prasad conjecture for $\text {U}(V_K)/\text {U}(V) $. arXiv preprint arXiv:2511.01301.

\bibitem[LMS]{LMS}Liu, D., Ma, J., Shi, F. (2025). Fourier–Jacobi models of Deligne–Lusztig characters and depth zero local descent for unitary groups. Israel Journal of Mathematics, 1-51.

\bibitem[LS]{LS}Lusztig, G., Srinivasan, B. (1977). The characters of the finite unitary groups. Journal of Algebra, 49(1), 167-171.

\bibitem[LW1]{LW1}Liu, D., Wang, Z. (2020). Descents of unipotent representations of finite unitary groups. Transactions of the American Mathematical Society, 373(6), 4223-4253.

\bibitem[LW2]{LW2}Liu, D., Wang, Z. (2021). Descents of unipotent cuspidal representations of finite classical groups. manuscripta mathematica, 165(1), 159-189.

\bibitem[LW3]{LW3}Liu, D., Wang, Z. (2021). On the Gan–Gross–Prasad problem for finite unitary groups. Mathematische Zeitschrift, 297(3), 997-1021.

\bibitem[MW]{MW}Moeglin, C., Waldspurger, J. L. (2012). La conjecture locale de Gross-Prasad pour les groupes sp\'eciaux orthogonaux: le cas g\'en\'eral. Astérisque, No. 347, 167-216.

\bibitem[Ree]{Ree}Reeder, M. (2007). On the restriction of Deligne-Lusztig characters. Journal of the American Mathematical Society, 20(2), 573-602.

\bibitem[Shi1]{Shi1}Shi, F. (2023). On the upper bound of the multiplicities of Fourier-Jacobi models over finite fields. Journal of Algebra, 634, 585-598.

\bibitem[Shi2]{Shi2}Shi, F. (2024). Periods of Deligne-Lusztig Characters associated to Spherical Varieties. arXiv preprint arXiv:2409.16853.

\bibitem[Wal1]{Wal1}Waldspurger, J.-L. (2010). Une formule intégrale reliée à la conjecture locale de Gross–Prasad. Compositio Mathematica, 146(5), 1180–1290. doi:10.1112/S0010437X10004744

\bibitem[Wal2]{Wal2}Waldspurger, J. L. (2012). Une formule int\'egrale reli\'ee\a la conjecture locale de Gross-Prasad, 2ème partie: extension aux repr\'esentations temp\'er\'ees I. Astérisque, No. 346(3), 171--312.

\bibitem[Wal3]{Wal3}Waldspurger, J. L. (2012). La conjecture locale de Gross-Prasad
pour les représentations tempérées. des groupes spéciaux orthogonaux.
Sur les conjectures de Gross et Prasad II. Astérisque. No. 347,
103-165.

\bibitem[Wang1]{Wang1}Wang, Z. (2021). On the Gan-Gross-Prasad problem for finite classical groups. Advances in Mathematics, 393, 108095.

\bibitem[Wang2]{Wang2}Wang, Z. (2025). Lusztig correspondence and the Gan-Gross-Prasad problem. Trans. Amer. Math. Soc., 378, 1283-1328.
\end{thebibliography}
\end{document}